\documentclass[letterpaper]{article}

\usepackage{amsmath,color,amsfonts,amsthm,amssymb,tikz,graphicx,url,hyperref,float,pgfplots,hyperref}
\usepackage[margin=1.2in]{geometry}
\pgfplotsset{compat=1.14}

\title{Maximum $\mathcal H$-free subgraphs}

\author{
	Dhruv Mubayi\footnote{Department of Mathematics, Statistics, and Computer Science, University of Illinois, Chicago, IL, 60607 USA. \texttt{Email:mubayi@uic.edu}. Research partially supported by NSF grants DMS-1300138 and 1763317.} $\qquad$ 
	Sayan Mukherjee\footnote{Department of Mathematics, Statistics, and Computer Science, University of Illinois, Chicago, IL, 60607 USA. \texttt{Email:smukhe2@uic.edu}.}
}

\newcommand{\Hh}{\mathcal H}
\newcommand{\B}{\mathcal B}
\newcommand{\ex}{\mbox{ex}}

\newcommand{\Aut}{\mbox{Aut}}
\newcommand{\twr}{\mbox{twr}}
\newcommand{\EIP}{\mathbf{EIP}}

\newtheorem{thm}{Theorem}[section]
\newtheorem{lemma}[thm]{Lemma}
\newtheorem{prop}[thm]{Proposition}
\theoremstyle{definition}
\newtheorem{claim}[thm]{Claim}
\newtheorem{prob}[thm]{Problem}
\newtheorem{defn}[thm]{Definition}

\newtheorem{cor}[thm]{Corollary}
\newtheorem*{notn}{Notation}

\begin{document}
\maketitle
\begin{abstract}
Given a family of hypergraphs $\mathcal H$, let $f(m,\mathcal H)$ denote the largest size of an $\mathcal H$-free subgraph that one is guaranteed to find in every hypergraph with $m$ edges. This function was first introduced by Erd\H{o}s and Koml\'{o}s in 1969 in the context of union-free families, and various other special cases have been extensively studied since then. In an attempt to develop a general theory for these questions,  we consider the following basic issue: which sequences of hypergraph families $\{\mathcal H_m\}$ have bounded $f(m,\mathcal H_m)$ as $m\to\infty$? A variety of bounds for $f(m,\mathcal H_m)$ are obtained which answer this question in some cases. Obtaining a complete description of sequences $\{\mathcal H_m\}$ for which $f(m,\mathcal H_m)$ is bounded seems hopeless.
\end{abstract}



\section{Introduction}
A hypergraph $H$ on vertex set $V(H)$ is a subset of $2^{V(H)}$. $H$ is an $\ell$-uniform hypergraph, or simply, an $\ell$-graph, if $H\subseteq\binom{V(H)}{\ell}$. All hypergraphs in this paper have finitely many vertices (and edges). Given a family of hypergraphs $\Hh$, a hypergraph $F$ is said to be $\Hh$-free if $F$ contains no copy of any member of $\Hh$ as a (not necessarily induced) subgraph. Given a hypergraph $F$ and a family $\Hh$, let $\ex(F,\Hh)$ be the maximum size of an $\Hh$-free subgraph of $F$. Define
\[
f(m,\Hh) := \min_{|F|=m}\ex(F,\Hh).
\]

Note that $f(m,\Hh)\ge c$ means that every $F$ with $m$ edges contains an $\Hh$-free subgraph $F'\subseteq F$ with $|F'|=c$. When the family $\Hh$ consists of a single hypergraph $H$, we abuse notation and write $f(m,H)$ instead of $f(m,\{H\})$.

This function was introduced by Erd\H{o}s and Koml\'{o}s in 1969 \cite{erdos-komlos1969}, who considered the case when $\Hh$ is the (infinite) family of hypergraphs $A,B,C$ with $A\cup B = C$. The problem was further studied by Kleitman \cite{kleitman-1973}, and later by Erd\H{o}s and Shelah \cite{erdos-shelah1972}, and finally settled by Fox, Lee and Sudakov \cite{fox-lee-sudakov2012} who proved that 
\[f(m,\Hh)=\left\lfloor\sqrt{4m+1}\right\rfloor - 1.\]

Erd\H{o}s and Shelah also considered the case when $\Hh$ is the family of hypergraphs $A_1,A_2,A_3,A_4$ with $A_1\cup A_2=A_3$ and $A_1\cap A_3=A_4$. They called this family $B_2$, proved that $f(m,B_2)\le (3/2)m^{2/3}$ and conjectured that this bound is asymptotically tight. This conjecture was settled by Bar\'{a}t, F\"{u}redi, Kantor, Kim and Patk\'{o}s in 2012 \cite{furedi-zoltan2012}, who also considered more general problems (see \cite{fox-lee-sudakov2012} for further work).

The same problem has been studied in the special case when $\Hh$ is a family of graphs. Let $f_2(m,\Hh)$ denote the maximum size of an $\Hh$-free subgraph that every graph with $m$ edges is guaranteed to contain. These investigations began with a question of Erd\H{o}s and Bollob\'{a}s \cite{erdos-bollobas1968} in 1966 about $f_2(m,C_4)$, followed up by a conjecture of Erd\H{o}s in \cite{erdos-unsolvedprob1971}. Consequently the problem of determining $f_2(m,H)$ for various graphs has received considerable attention in the recent years \cite{foucaud-krivelevich2015,conlon-fox-sudakov2014,conlon-fox-sudakov2016}. The authors of \cite{conlon-fox-sudakov2014,conlon-fox-sudakov2016} also considered the problem in the case of $\ell$-graphs.

In the hope of obtaining a general theory for these problems, we investigate the following basic question:

\begin{equation}
\label{eqn:mainqn}
\mbox{For which sequence of families }\{\Hh_m\}_{m=1}^\infty\mbox{ is }f(m,\Hh_m)\mbox{ bounded (as }m\to\infty\mbox{)?}
\end{equation}

\smallskip
Question (\ref{eqn:mainqn}) is too general to solve completely, so we focus on special cases. In subsection 2.1 we state our results for constant $\{\Hh_m\}_{m=1}^\infty$, and in subsection 2.2 we consider non-constant $\{\Hh_m\}_{m=1}^\infty$.


\section{Our Results}


\subsection{Constant Sequences}
Suppose $\{\Hh_m\}_{m=1}^\infty$ is a sequence such that $\Hh_m=\Hh$ for every $m$. First, we note that if $\Hh$ consists of finitely many members, then the answer to Question (\ref{eqn:mainqn}) is given by the following characterization. A $q$-sunflower is a hypergraph $\{A_1,\ldots,A_q\}$ such that $A_i\cap A_j = \bigcap_{s=1}^q A_s$ for all $i\neq j$. This common intersection is referred to as the \emph{core} of the sunflower.

\begin{thm}
\label{thm:sunflower}
	Fix a family of hypergraphs $\Hh$ with finitely many members. If $\Hh$ contains a $q$-sunflower with sets of equal size, then $f(m,\Hh)\le q-1$. Otherwise, $f(m,\Hh)\to\infty$ as $m\to\infty$.
\end{thm}

Next, in the same spirit as the properties of being union-free and having no $B_2$, if the (infinite) family $\Hh$ specifies the intersection type of $k$ sets (ie whether they are empty or not), then a characterization can be obtained in the form of Theorem \ref{thm:binaryinfo}. Before stating the theorem, we first define what we call an even hypergraph and an $\ell$-uneven hypergraph. A $k$-edge hypergraph is a hypergraph with $k$ edges.

\begin{defn}[Even and $\ell$-uneven hypergraphs]
A $k$-edge hypergraph $H=\{A_1,\ldots,A_k\}$ is said to be \emph{even} if for every $1\le\ell\le k$ and for every $I,J\in \binom{[k]}{\ell}$, $\bigcap_{i\in I}A_i=\varnothing\iff \bigcap_{j\in J}A_j=\varnothing$. It is said to be \emph{$\ell$-uneven} if there exist $I,J\in \binom{[k]}{\ell}$ such that $\bigcap_{i\in I}A_i\neq\varnothing$ but $\bigcap_{j\in J}A_j=\varnothing$.
\end{defn}
\begin{thm}
\label{thm:binaryinfo}
Let $1\le \ell<k$. Let $\Hh$ be the (infinite) family of all $\ell$-uneven $k$-edge hypergraphs. Then, $f(m,\Hh)\to\infty$ as $m\to\infty$. Conversely, if $\Hh$ is the family of all even $k$-edge hypergraphs, we have $f(m,\Hh) = k-1$.
\end{thm}

\subsection{Non-constant Sequences}
As a first step towards understanding the general problem in (\ref{eqn:mainqn}), we focus on the case when for every $m\ge 1$, $\Hh_m=\{H_m\}$ for a single hypergraph $H_m$, and further assume that all these hypergraphs $H_m$ have the same number of edges. Thus we ask the following question:

\begin{equation}
\label{eqn:mainqn2}
\mbox{For which sequence of }k\mbox{-edge hypergraphs }\{H_m\}_{m=1}^\infty\mbox{ is }f(m,H_m)\mbox{ bounded (as }m\to\infty\mbox{)?}
\end{equation}

\smallskip
We are unable to answer question (\ref{eqn:mainqn2}) completely, even for $k=3$. Our main results provide several necessary, or sufficient conditions that partially answer (\ref{eqn:mainqn2}). Before presenting them, we introduce the following crucial definition:

\begin{defn}[Equal Intersection Property]
	 For $k\ge 2$, Let $\EIP_k$ denote the set of all $k$-edge hypergraphs $H=\{A_1,\ldots,A_k\}$ such that for every $1\le \ell \le k$ and $I,J\in \binom{[k]}{\ell}$, we have $\left|\bigcap_{i\in I}A_i\right| = \left|\bigcap_{j\in J}A_j\right|$.
\end{defn}
Every $H=\{A_1,\ldots,A_k\}\in \EIP_k$ can be encoded by $k$ parameters $(b_1,\ldots,b_k)$, corresponding to the $k$ distinct sizes appearing in the Venn diagram of $H$. More precisely, for $1\le \ell\le k$, and for all $I\in \binom{[k]}\ell$,
\[
b_\ell := \left|\bigcap_{i\in I}A_i\setminus \bigcup_{i\in [k]\setminus I}A_i\right|.
\]

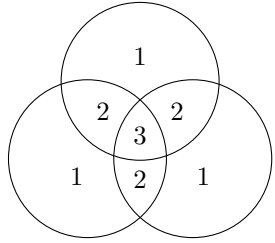
\begin{figure}[h]
\begin{center}
	\begin{tikzpicture}[scale=0.7]
	\draw (0,0) circle [radius=1.5]; 
	\draw (2,0) circle [radius=1.5];
	\draw (1,1.47) circle [radius=1.5];
	\draw (-0.2,0) node [below]{$1$};
	\draw (2.2,0) node [below]{$1$};
	\draw (1,1.6) node [above] {$1$};
	\draw (1,0.45) node {$3$};
	\draw (0.3,0.9) node {$2$};
	\draw (1.7,0.9) node {$2$};
	\draw (1,-0.4) node {$2$};
	\end{tikzpicture}
	\caption{An example: $H(1,2,3)\in\mathbf{EIP}_3$}
\end{center}
\end{figure}

By inclusion-exclusion, $b_1,\ldots,b_k$ are well-defined for hypergraphs in $\EIP_k$. We denote $H\in \EIP_k$ with parameters $b_1,\ldots,b_k\ge 0$ by $H(\vec{b})$, where $\vec b = (b_1,\ldots,b_k)$. We shall see later (Lemma \ref{lem:EIP}) that every sequence of $k$-edge hypergraphs $\{H_m\}$ such that $f(m,H_m)$ is bounded, can only have finitely many members not in $\EIP_k$. For sequences $\{H_m\}_{m=1}^\infty$ such that $H_m\in\EIP_k$ for every $m\ge 1$, we obtain a sequence of length $k$ vectors $\{\vec{b}(m)\}_{m=1}^\infty$, where $\vec{b}(m)=(b_1(m),\ldots,b_k(m))$. We use boldface and write $\vec{\bf b}$ for the sequence $\{\vec{b}(m)\}_{m=1}^\infty$. Since every two edges form a $2$-sunflower, this discussion together with the construction from Theorem \ref{thm:sunflower} gives us

\begin{prop}
\label{prop:prop_k=2}
Let $H_m$ be a 2-edge hypergraph for each $m\ge 1$. Then $f(m, H_m)$ is bounded as $m\rightarrow \infty$ if and only if $H_m \in \EIP_2$ for all but finitely many $m$.
\end{prop}

We may therefore assume in what follows that $ k \ge 3$.

\begin{defn}[$\alpha(\vec{\bf b})$]
	For every sequence of length $k$ vectors $\vec{\bf b}=\{\vec{b}(m)\}_{m=1}^\infty$ and $m\ge 1$, let
	\[
	\alpha(\vec{\bf b}) (m) := \min_{1\le i \le k-2}\left(\frac{b_i(m)}{mb_{i+1}(m)}\right).
	\]
\end{defn}

Now we state our main results. To simplify notation we will often write $b_i$ instead of $b_i(m)$ and $\alpha(\vec{\bf b})$ instead of $\alpha(\vec{\bf b})(m)$.
\medskip

\begin{thm}
\label{thm:generalupperlowerbd}
	Let $k\ge 3$. Suppose the sequence of length $k$ vectors $\vec{\bf b}$ satisfies $b_1,\ldots,b_{k-2}>0$, $b_{k-1},b_k\ge 0$ for every $m$. Then, for $m\ge 6$,
	\[
	\left(\frac{1}{2\left(\alpha(\vec{\bf b})+\frac 1m\right)\binom{b_{k-1}+b_k}{b_k}}\right)^{\frac 1k}\le f(m,H(\vec{\bf b}))\le \frac{k(k-1)}{\alpha(\vec{\bf b})}+k-1.
	\]
\end{thm}

Theorem \ref{thm:generalupperlowerbd} implies that when $\binom{b_{k-1}+b_k}{b_k}$ is bounded from above, $f(m,H(\vec{\bf b}))$ is bounded from above if and only if the sequence $\alpha(\vec{\bf b})$ is bounded away from zero.

We also have the following additional lower bound on $f(m,H(\vec{\bf b}))$:

\begin{thm}
\label{thm:generallowerbd}
	Fix $k\ge 3$. Let $\vec{\bf b}=\{\vec b (m)\}_{m=1}^\infty$ be such that $b_k(m)=b_k$ for every $m$. Then, for $m\ge 6$,
	\[
	f(m,H(\vec{\bf b}))\ge
	\left\{ 
		\begin{array}{lc}
		m^{\frac 1{k(b_k+1)}}\left(\frac{b_{k-1}}{4(b_{k-2}+2b_{k-1})}\right)^{\frac 1k}, & k\ge 4,\\
		m^{\frac1{b_3+2}}\left(\frac{b_2}{4(b_1+2b_2)}\right)^{\frac{b_3+1}{b_3+2}}, & k=3.
		\end{array}
	\right.
		\]
\end{thm}\medskip

We now focus on  $k=3$.
In this case  $\alpha(\vec{\bf b})=b_1/mb_2$ and Theorem \ref{thm:generalupperlowerbd} reduces to 
\begin{equation} \label{reduces}
\left(\frac{1}{2\left(\frac{b_1}{mb_2}+\frac 1m\right)\binom{b_{2}+b_3}{b_3}}\right)^{\frac 13}\le f(m,H(\vec{\bf b}))\le \frac{6mb_2}{b_1}+2.
\end{equation}
When $b_3=0$, (\ref{reduces}) implies that $f(m, H_3(b_1, b_2, 0))$ is bounded if and only if
 $b_1 = \Omega(mb_2)$. We now turn to $b_3=1$ which already seems to be a very interesting special case that is related to an open question in extremal graph theory (see Problem~\ref{prob:prob2} in Section~\ref{problems}). 
Here (\ref{reduces}) and Theorem~\ref{thm:generallowerbd} yield the following.
\begin{cor} \label{cor:k=3}
Let $m \rightarrow \infty$. 	Then $f(m,H_3(b_1,b_2,1))$ is bounded when $b_1=\Omega(m b_2)$
and it is unbounded when either $b_1+b_2 = o(m)$ or $b_1=o(\sqrt m\,b_2)$.
\end{cor}
Corollary~\ref{cor:k=3} can be summarized in Figure \ref{fig:initialpic}. The light region corresponds to a bounded $f(m,H(\vec{\bf b}))$, and the dark region corresponds to unbounded $f(m,H(\vec{\bf b}))$. White regions correspond to areas where we do not know if $f(m,H(\vec{\bf b}))$ is bounded or not.

\begin{figure}[H]
\begin{center}
	\begin{tikzpicture}[scale=1.7]
		\draw  (0,4) node (yaxis) [above] {$b_1$}
        	|- (4,0) node (xaxis) [right] {$b_2$};
		\draw [red] (0,0) coordinate (o) -- (2.5,4) coordinate (upper);
		\draw (o) -- (4,2.5) coordinate (lower);
		\draw (upper) node [above] {$b_1=mb_2$};
		\draw (lower) node [right] {$b_1=\sqrt m\, b_2$};
		\draw [dashed] (0,1.4) -- (0.875,1.4) coordinate (p);
		\draw (p)--(2.24,1.4) coordinate (q);
		\draw [dashed] (q)--(4,1.4);
		\draw (4,1.4) node [right] {$b_1=m$};
		\fill [red,opacity=0.2] (o)--(2.25,3.6)--(0,3.6)--cycle;
		\fill [blue, opacity=0.5] (o)--(3.6,2.25)--(3.6,0)--cycle;
		\fill [blue, opacity=0.5] (o)--(p)--(q)--cycle;
		\draw (0.9,3) node {2.7};
		\draw (3.25,1.7) node {2.8};
		\draw (3.02,1.1) node {2.7, 2.8};
		\draw (1.1,1.1) node {2.7};
	\end{tikzpicture}
	\caption{Theorems \ref{thm:generalupperlowerbd} and \ref{thm:generallowerbd} for $\vec{\bf b}=(b_1,b_2,1)$}
	\label{fig:initialpic}
\end{center}
\end{figure}
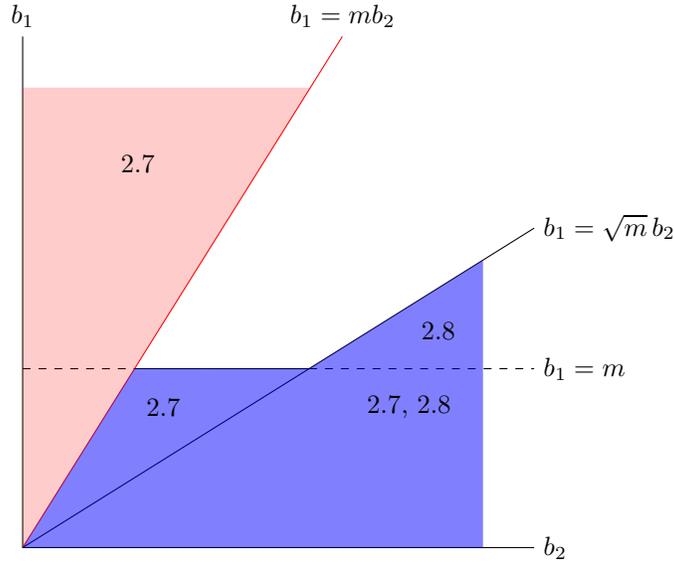

We are able to refine our results slightly via the following result.
\begin{thm}\label{thm:inversiveplane}
	For every odd prime power $q$ we have
	$f(q^2+1,H(q^2-q-1,q,1))=2.
	$\end{thm}
As a consequence, we obtain the following:
\begin{cor}
\label{cor:extendinversive}
	When $b_1\ge b_2^{\,2}-b_2-1$ and $b_2^{\,2}\ge m$, we have $f(m,H_3(b_1,b_2,1))=2$.
\end{cor}

Corollary \ref{cor:extendinversive} yields the following improvement on Figure \ref{fig:initialpic}. Note that we are using the parabola $b_1=b_2^{\,2}$ as an asymptotic approximation of Corollary \ref{cor:extendinversive}. Therefore on this parabola, $f(m,H_3)=2$. We also know, by virtue of Theorem \ref{thm:necessary_condition_for_f=2}, that in the white region to the right of the parabola and between the two lines, we have $f(m,H_3(b_1,b_2,1))>2$.

\begin{figure}[H]
\begin{center}
	\begin{tikzpicture}[scale=0.9]
		\draw (0,9) node (yaxis) [above] {$b_1$}
            |-(9,0) node (xaxis) [right] {$b_2$};
		\draw [red] (0,0) coordinate (o) -- (5,8) coordinate (upper);
		\draw (o) -- (8,5) coordinate (lower);
		\draw (upper) node [right] {$b_1=mb_2$};
		\draw (lower) node [right] {$b_1=\sqrt m b_2$};
		\fill [red,opacity=0.2] (o)--(4.5,7.2)--(0,7.2)--cycle;
		\fill [blue, opacity=0.5] (o)--(7.2,4.5)--(7.2,0)--cycle;
    	\draw [red, domain=10/9:4.5] plot (\x,9*\x^1.8/16);
		\draw [dashed, domain=0:10/9] plot (\x,9*\x^1.8/16);
		\draw (4.5,8.5) node [left] {$b_1=b_2^{\,2}$};
		\coordinate (invplane) at (10/9,50/72);
		\coordinate (critical) at (3.694,5.9104);
		\coordinate (critical2) at (10/9, 16/9);
		\draw [red] (critical2) -- (invplane);
		\draw [dashed] (invplane) -- (10/9,0);
		\draw (10/9,0) node [below] {$\sqrt m$};
		\begin{scope}
        	\clip [domain=10/9:2.85] plot (\x,9*\x^1.8/16)--(critical)--(critical2)--cycle;
        	\fill [red,opacity=0.2] (invplane)--(4,2.5)--(4.5,7.2)--(10/9,16/9)--cycle;
		\end{scope}
		\coordinate (left) at (0,50/72);
		\coordinate (critical3) at (250/576,50/72);
		\draw [dashed] (left)--(critical3);
		\draw (critical3)--(invplane);
		\draw (0,50/72) node [left] {$m$};
		\fill [blue,opacity=0.5] (invplane)--(critical3)--(o)--cycle;
		\draw [red,fill] (invplane) circle [radius=0.13];
		\draw [dashed] (critical3)--(250/576,0);
		\draw (250/576,-0.07) node [below] {$1$};
		\draw [dashed] (critical)--(3.694,0);
		\draw (3.694,-0.1) node [below] {$m$};
		\draw (1.9,6) node {$2.7$};
		\draw (6,1.9) node {$2.8$};
		\draw (1.47,1.75) node {$2.10$};
	\end{tikzpicture}
	\caption{$\vec{\bf b}=(b_1,b_2,1)$}
	\label{fig:finalpic}
\end{center}
\end{figure}
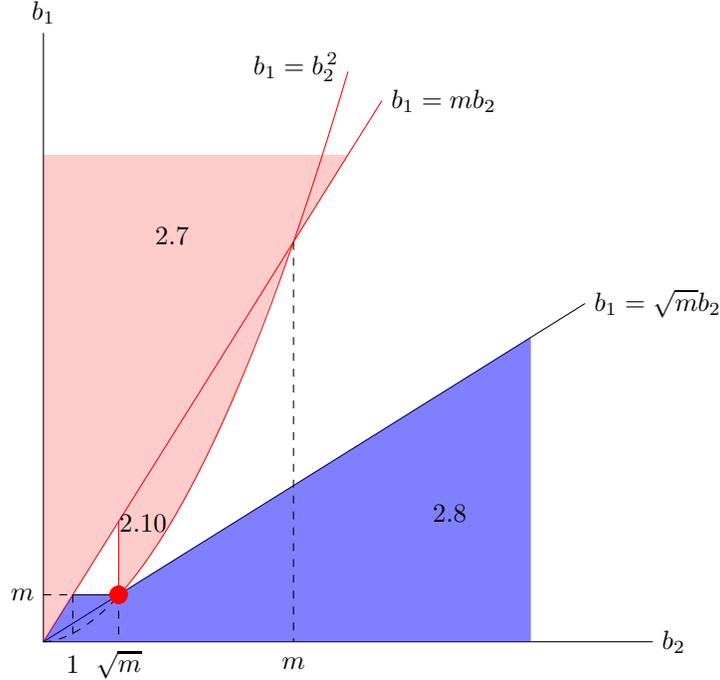


\section{Proofs of Theorems \ref{thm:sunflower} and \ref{thm:binaryinfo}}

In this section, we prove Theorems \ref{thm:sunflower} and \ref{thm:binaryinfo}, which answer question (\ref{eqn:mainqn}) for constant sequences. We use the following well-known facts about sunflowers and diagonal hypergraph Ramsey numbers. 

Recall that a $q$-sunflower is a hypergraph $\{A_1,\ldots,A_q\}$ such that $A_i\cap A_j = \bigcap_{s=1}^q A_s$. The celebrated Erd\H{o}s-Rado sunflower Lemma~\cite{erdos-rado-sunflower1960} states the following.
\begin{lemma}[Erd\H{o}s-Rado]
\label{lem:erdosradosunflower}
Let $H$ be an $r$-graph with $|H|=r!(\alpha-1)^r$. Then, $H$ contains an $\alpha$-sunflower.
\end{lemma}

Next, recall that the hypergraph Ramsey number $r_\ell(s,t)$ is the minimum $N$ such that any $\ell$-graph on $N$ vertices, admits a clique of size $s$ or an independent set of size $t$. The following is a well-known theorem of Erd\H{o}s, Hajnal and Rado \cite{erdos-hajnal-rado-partitionrelations1965}:
\begin{thm}
\label{thm:ramseythm}
There are absolute constants $c(\ell),c'(\ell)$ such that
\[
\twr_{\ell-1}(c't^2)< r_\ell(t,t) < \twr_\ell(ct).
\]
Here the tower function $\twr_k(x)$ is defined by $\twr_0(x)=1$ and $\twr_{i+1}(x)=2^{\twr_i(x)}$.
\end{thm}

The right side of this theorem can be rewritten as follows:

\begin{equation}
	\label{eqn:ramseylemma}
	\begin{array}{l}\mbox{Let }F\mbox{ be any }\ell\mbox{-graph on }n\mbox{ vertices. Then there is an absolute constant }c_\ell\mbox{ such that there is}\\
	\mbox{a subgraph }F'\subset F\mbox{ with }|V(F')|\ge c_\ell\cdot \log_{(\ell)}(n)\mbox{, which is either a clique or an independent}\\
	\mbox{set. Here }\log_{(\ell)}\mbox{ denotes iterated logarithms.}
	\end{array}
\end{equation}
\medskip
 
Now we are prepared to prove Theorems \ref{thm:sunflower} and \ref{thm:binaryinfo}. Recall that a hypergraph is uniform if all its edges have the same size, otherwise it is non-uniform.

\medskip
\begin{proof}[Proof of Theorem \ref{thm:sunflower}]
	
	Fix a family of hypergraphs $\Hh$ with $n$ members, $\Hh=\{H_1,\ldots,H_n\}$. Let $H_i\in \Hh$ be an $r$-uniform $q$-sunflower with core $W$. For every $m\ge q$, let $F$ be an $r$-uniform $m$-sunflower with core $W$. Then every subset of $F$ of size $q$ is isomorphic to $H_i$, thus proving $f(m,\Hh)\le q-1$.
	
	Now, suppose $\Hh$ contains $\ell$ many uniform hypergraphs labelled $H_1,\ldots, H_\ell$, and $(n-\ell)$ many non-uniform hypergraphs labelled $H_{\ell+1},\ldots,H_n$. For $1\le i\le \ell$, let $r_i$ be the uniformity of $H_i$. Given any hypergraph $F$ with $m$ edges, we find a large $\Hh$-free subgraph as follows. First, since $H_n$ is non-uniform, it contains a set of size $a$ and a set of size $b\neq a$. Clearly, at least half of the edges of $F$ have size $\neq a$, or at least half of them have size $\neq b$. Take the appropriate subgraph $F_1\subset F$ of size $\ge \frac m2$. By successively halving the sizes, we obtain a chain of hypergraphs $F_{n-\ell}\subset F_{n-\ell-1}\subset \cdots \subset F_1\subset F$ such that $F_{n-\ell}$ is $\{H_{\ell+1},\ldots, H_n\}$-free, and $|F_{n-\ell}|\ge \frac m{2^{n-\ell}}$.
	
	We now deal with the uniform part of $\Hh$. Notice that by Lemma \ref{lem:erdosradosunflower}, any $r$-graph $G$ with $|G|=m$ contains an $\alpha$-sunflower, as long as $m>r!\alpha^r$. Taking $\alpha = \left\lfloor c_r m^{1/r}\right\rfloor$ where $c_r = \left((2r)!\right)^{-1/r}$, satisfies the required condition. So, every $r$-graph $G$ of size $m$ contains a sunflower of size $\left\lfloor c_r m^{1/r}\right\rfloor$.
	
	Since $H_\ell$ is $r_\ell$-uniform, we note that either $F_{n-\ell}$ contains a subgraph of size $\frac 12 |F_{n-\ell}|$ which has no sets of size $r_\ell$ (and hence is $H_\ell$-free), or there is a subgraph of size $\frac 12 |F_{n-\ell}|$ which is $r_\ell$-uniform. In the second case, using Lemma \ref{lem:erdosradosunflower} on this subgraph, we obtain an $H_\ell$-free subgraph of $F_{n-\ell}$ of size at least
	\[
	\min\left\{\frac m{2^{n-\ell+1}}\ ,\  c_{r_\ell}\left(\frac m{2^{n-\ell+1}}\right)^{\frac1{r_\ell}}\right\}\ge c'_\Hh \cdot m^{\frac 1{r_\ell}}.
	\]
	
	We iterate the same argument $\ell-1$ more times, to finally obtain a constant $C_\Hh$ and a subgraph $F_\ell\subset F$ such that $F_\ell$ is $\Hh$-free, and
	\[
	|F_\ell|\ge C_\Hh\cdot m^{\frac 1{r_1\ldots r_\ell}}.
	\]
\end{proof}

\medskip

\begin{proof}[Proof of Theorem \ref{thm:binaryinfo}]

	Let $F=\{F_1,\ldots,F_m\}$ have size $m$. Suppose $\Hh$ is $\ell$-uneven for some $1\le \ell <k$. Then, there are distinct subsets $I,J\in\binom{[k]}{\ell}$, such that $\bigcap_{i\in I}A_i = \varnothing$ and $\bigcap_{j\in J}A_j\neq\varnothing$ for every $H=\{A_1,\ldots,A_k\}\in\Hh$. Then, we construct an $\ell$-graph $G$ with vertex set $F$, and hyperedges $\left\{\{F_1,\ldots,F_\ell\}:\ F_1\cap\cdots\cap F_\ell=\varnothing\right\}$. By (\ref{eqn:ramseylemma}), there is a a constant $c_\ell$ and a subset $F'\subseteq F$ of size $\ge c_\ell\cdot\log_{(\ell)}(m)$, such that $F'$ is either a clique or an independent set in $G$. In either case, $F'$ is $\Hh$-free.
	
	On the other hand, if $\Hh$ is such that that for every $1\le \ell\le k$ and $I,J \in\binom{[k]}\ell$, $\bigcap_{i\in I} A_{i}\neq \varnothing \iff \bigcap_{j\in J}A_{j}\neq \varnothing$, then there exists $L$ such that for every $1\le \ell\le L$, the $\ell$-fold intersections $A_{i_1}\cap\cdots\cap A_{i_\ell}$ are nonempty, whereas for $L<\ell\le k$, the $\ell$-fold intersections $A_{i_1}\cap\cdots\cap A_{i_\ell}$ are all empty for $i_1,\ldots,i_\ell\in [k]$. 
	
	For every $m\ge k$, we construct a hypergraph $F=\{F_1,\ldots,F_m\}$ in the following manner. Consider the bipartite graph $B=\left([m],\binom{[m]}{L}\right)$ where $x\in [m]$ is adjacent to $y\in\binom{[m]}{L}$ iff $x\in y$. Let $F_i$ be the set of neighbors in $B$ of the vertex $i\in [m]$. Notice that for any $i_1,\ldots, i_\ell$,
\[
	F_{i_1}\cap\ldots\cap F_{i_\ell} = \left\{\begin{array}{lc}\varnothing, & \ell > L,\\ \neq \varnothing, & \ell \le L.\end{array}\right.
\]
	This construction therefore shows that $f(m,\Hh)= k-1$.
\end{proof}


\section{Proof of Theorem \ref{thm:generalupperlowerbd}}
In this section, we prove Theorem \ref{thm:generalupperlowerbd}. We begin with some preliminary exploration of the family $\EIP_k$.

First, we make the crucial observation regarding question (\ref{eqn:mainqn2}) that every sequence of $k$-edge hypergraphs $\{H_m\}$ such that $f(m,H_m)$ is bounded, can only have finitely many members not in $\EIP_k$. This follows immediately from Lemma \ref{lem:EIP}. Furthermore, for any $H(\vec{b})\in\EIP_k$, one can explicitly determine the relation between the intersection sizes and the parameters $b_1,\ldots, b_k$ by inclusion-exclusion. We state this relation in Lemma \ref{lem:a_to_b}.

\begin{lemma}
	\label{lem:EIP}
	Suppose $H=\{A_1,\ldots, A_k\}$ satisfies the following for some $1\le \ell\le k$: there are two sets of indices $I,J\in \binom{[k]}\ell$ such that $|\bigcap_{i\in I}A_{i}|=a$ and $|\bigcap_{j\in J}A_{j}|=b$ with $a\neq b$. Then there is a constant $c_\ell$ such that $f(m,H)\ge c_\ell\cdot \log_{(\ell)}(m)$.
\end{lemma}
\begin{proof}
	Let $F$ be any hypergraph with $m$ edges. Construct an $\ell$-graph $G$ with $F$ as its vertex set, and hyperedges 
	\[
	\left\{\{B_1,\ldots,B_\ell\}: |B_1\cap\cdots\cap B_\ell|=a\right\}.
	\]
	By (\ref{eqn:ramseylemma}), there exists a subset $F'\subseteq F$ of size $c_\ell\cdot\log_{(\ell)}(m)$ which is either a clique or an independent set in $G$. In either case, $H$ cannot be contained in $F'$.
\end{proof}

Lemma \ref{lem:EIP} implies that if there are infinitely many $m$ such that $H_m\not\in\EIP_k$, then for each such non-EIP hypergraphs we have $f(m,H_m)\ge c'\cdot \log^{(k)}(m)$, where $c'$ is the absolute constant $c'=\min\{c_1,\ldots,c_k\}$. This is an infinite subsequence of $\{H_m\}$. Therefore, if $f(m,H_m)$ is bounded, then by looking at the tail of $\{H_m\}$, we may assume WLOG that $H_m\in\EIP_k$ for every $m\ge 1$.

Recall that hypergraphs $H\in\EIP_k$ are characterized by the length $k$-vector $\vec{b}$, and for every sequence of hypergraphs $\{H_m\}_{m=1}^{\infty}$, we have a corresponding sequence of length $k$ vectors $\vec{\bf b}$.

We now state the relation between the intersection sizes and the parameters $b_1,\ldots, b_k$ for $H(\vec{b})\in \EIP_k$.

\begin{lemma}
	\label{lem:a_to_b}
	Let $H(\vec b)\in \EIP_k$, and $a_i=|A_1\cap\cdots \cap A_i|$, for each $1\le i\le k$. Then,
\begin{equation}
\label{eqn:a=sum(b)}
b_i = a_i - \binom{k-i}{1}a_{i+1}+\binom{k-i}{2}a_{i+2}-\cdots + (-1)^{k-i}\binom{k-i}{k-i}a_k.
\end{equation}
\end{lemma}
\hfill{$\square$}


Before proving Theorem \ref{thm:generalupperlowerbd}, we prove an auxiliary upper bound in Lemma \ref{lem:linearmatrixupperbd}, which provides a better upper bound on $f(m,H(\vec{\bf b}))$ with tighter constraints on $\vec{\bf b}$. 

\begin{lemma}
	\label{lem:linearmatrixupperbd}
	Suppose $\vec{\bf b}=(b_1,\ldots,b_k)$ is such that $b_i\ge 0$, and for every $1\le i\le k-1$,
	\begin{equation}
	\label{eqn:cond_on_bi}
	\sum_{j=i}^{k-1}(-1)^{j-i}\binom{m-k+j-i-1}{j-i}b_j\ge 0.
	\end{equation}
	
	Then $f(m,H(b_1,\ldots,b_k))=k-1$.
\end{lemma}

\noindent{\it Proof of Lemma \ref{lem:linearmatrixupperbd}.} Let $\vec{\bf b}$ satisfy the restrictions given in (\ref{eqn:cond_on_bi}). Note that we need to construct a hypergraph sequence $\{F_m\}_{m=1}^\infty$, such that every $k$-edge subgraph of $F_m$ is isomorphic to $H(\vec{b})$. To achieve this, we define the following general construction:
\medskip

\paragraph{{\bf Construction $F_m^{d_1,\ldots,d_k}$:}}

Given $d_1,\ldots, d_k\ge 0$ and $m\ge k$, let $B=([m],Y)$ be the bipartite graph with parts $[m]$ and $Y$, where $Y$ is defined as follows.
For $1\le\ell\le k$ and $1\le j\le d_\ell$, let
\[
Y_j^\ell =	\left\{
\begin{array}{cc} \{ v_j^S: S \in {[m] \choose \ell} \}, & \ \ell < k \\
\{ w_j\}, &\ \ell = k
\end{array}
\right\}
\]
and suppose that $v_j^S\neq v_{j'}^{S'}$ for every $(j,S)\neq(j',S')$. Then
\[
Y=\bigcup_{\ell=1}^k\bigcup_{j=1}^{d_\ell}Y_j^\ell.
\]
For $x\in[m]$ and $v_j^S\in Y$, let $(x,v_j^S)\in E(B)$ iff $x\in S$, and let $(x,w_j)\in E(B)$ for every $x\in [m]$ and $w_j\in Y$. Then, define $F_m^{d_1,\ldots,d_k}=\{A_1,\ldots,A_m\}$, where $A_i = N_B(i)\subset Y$ for $i=1,\ldots,m$.
\hfill{$\blacksquare$}

\medskip		

\noindent For example, the construction $F_4^{1,2,3}$ is given by:
\[
\left\{
\begin{aligned}
\ A_1&=\left\{v^1_1; v^{12}_1,v^{12}_2, v^{13}_1,v^{13}_2, v^{14}_1,v^{14}_2; w_1, w_2, w_3\right\}\\
\ A_2&=\left\{v^2_1; v^{12}_1,v^{12}_2, v^{23}_1,v^{23}_2, v^{24}_1,v^{24}_2; w_1, w_2, w_3\right\}\\
\ A_3&=\left\{v^3_1; v^{13}_1,v^{13}_2, v^{23}_1,v^{23}_2, v^{34}_1,v^{34}_2; w_1, w_2, w_3\right\}\\
\ A_4&=\left\{v^4_1; v^{14}_1,v^{14}_2, v^{24}_1,v^{24}_2, v^{34}_1,v^{34}_2; w_1, w_2, w_3\right\}
\end{aligned}\right\}.
\]
Informally, $A_i$ consists of one vertex $v_1^i$ corresponding to $\{i\}$, two vertices $v_1^{ij}$ and $v_2^{ij}$ corresponding to two-element subsets $\{i,j\}$, and three vertices $w_1,w_2,w_3$ that are in the common intersection of all the $A_i$'s, $1\le i\le 4$.

\medskip
We observe the following property of the intersection sizes of the edges of $F_m^{d_1,\ldots,d_k}$.
\begin{claim}
	\label{clm:a=sum(f)}
	$F_m^{d_1,\ldots,d_k}\in \EIP_k$. Furthermore, for $1\le i\le k$ and any $i$-edge subgraph $\{A_{r_1},\ldots, A_{r_i}\}\subset F_m^{d_1,\ldots,d_k}$, the size of the common intersection $a_i:=|A_{r_1}\cap\cdots\cap A_{r_i}|$ is given by
	\begin{equation}
	\label{eqn:a=sum(f)}
	a_i=d_i+\binom{m-i}{1}d_{i+1}+\cdots+\binom{m-i}{k-1-i}d_{k-1}+d_k.
	\end{equation}
\end{claim}
\noindent {\it Proof of Claim \ref{clm:a=sum(f)}}. 
Suppose $G=\{A_{r_1},\ldots,A_{r_i}\}\subset F_m^{d_1,\ldots,d_k}$. We shall now count $|A_{r_1}\cap\cdots\cap A_{r_i}|$. For a fixed hypergraph $F_m^{d_1,\ldots,d_k}\supseteq G'\supseteq G$, let $U_{G'}$ denote the set of all vertices of $F_m^{d_1,\ldots,d_k}$ which are in all the edges of $G'$ but none of the edges of $F_m^{d_1,\ldots,d_k}\setminus G'$. Notice that $A_{r_1}\cap\cdots\cap A_{r_i}$ is a disjoint union of $U_{G'}$'s, $G'\supseteq G$. Therefore,
\begin{equation}
\label{eqn:proofclaimeq}
|A_{r_1}\cap\cdots\cap A_{r_i}| = \sum_{G'\supseteq G}|U_{G'}| = \sum_{G'\supseteq G}\left|\bigcap_{X\in G'} X\setminus \bigcup_{X\not\in G'} X\right|.
\end{equation}
Fix a $G'\supseteq G$. Let $G'=\{A_{r_1},\ldots,A_{r_i},A_{s_1},\ldots,A_{s_{|G'|-i}}\}$. We observe that,
\begin{itemize}
	\item For $i\le |G'|<k$, $U_{G'}$ consists exactly of the vertices $\left\{v_j^{\{r_1,\ldots,r_i,s_1,\ldots,s_{|G'|-1}\}}: 1\le j\le d_{|G'|}\right\}$.
	\item For $k\le |G'|<m$, $\bigcap_{X\in G'}X = \{w_1,\ldots, w_{d_k}\} \subseteq \bigcup_{X\not\in G'}X$, thus $U_{G'}=\varnothing$.
	\item For $|G'|= m$, $U_{G'} = \bigcap_{X\in G'}X = \{w_1,\ldots, w_{d_k}\}$.
\end{itemize}
Therefore,
\[
|U_{G'}|=\left\{\begin{array}{cc}
d_{|G'|}, & i\le|G'|< k,\\
0,        & k\le|G'|<m,\\
d_k,      & |G'|=m.
\end{array}\right.
\]
Plugging back these values into (\ref{eqn:proofclaimeq}), we get
\[
a_i=d_i+\binom{m-i}{1}d_{i+1}+\cdots+\binom{m-i}{k-1-i}d_{k-1}+d_k
\]
for every $1\le i \le k$.
\hfill{$\blacksquare$}

\medskip

Now we return to the proof of Lemma \ref{lem:linearmatrixupperbd}. Given a length $k$ vector $\vec b\ge 0$ which satisfies (\ref{eqn:cond_on_bi}) for $1\le i\le k-1$, let $d_i$ be the left hand side of (\ref{eqn:cond_on_bi}), i.e.,
\[
d_i:=\sum_{j=i}^{k-1}(-1)^{j-i}\binom{m-k+j-i-1}{j-i}b_j,
\]
and let $d_k=b_k$. Now, we look at the construction $F_m = F_m^{d_1,\ldots,d_k}$, and pick a $k$-edge subgraph $G\subset F_m$. Observe that $G\in\EIP_k$, and therefore there is a length $k$ vector $\vec{g}$ such that $G = H(\vec{g})$. It suffices to check that $\vec{g}=\vec{b}$.

\medskip

Suppose $G=\{A_1,\ldots,A_k\}$. For $1\le i\le k$, let $a_i:=|A_1\cap\cdots\cap A_i|$. Recall that Lemma \ref{lem:a_to_b} gave us a way of computing $\vec{g}$ in terms of $\vec{a}$, and Claim \ref{clm:a=sum(f)} computes $\vec{a}$ in terms of $\vec{d}$. In order to precisely write down these relations, we introduce a few matrices.

\begin{notn}
	$\qquad$
	\begin{itemize}
		\item Let $a^{(m)}_{ij}=\binom{m-i}{j-i}$ and $b^{(m)}_{ij}=(-1)^{j-i}\binom{m-i}{j-i}$.\footnote{ By our convention, $\binom xy=0$ if  $y<0$. Thus $a^{(m)}_{ij}=b^{(m)}_{ij}=0$ whenever $j<i$.} Then, we denote by $A_{k,m}$ and $B_{k,m}$ the upper triangular matrices
		\[
		A_{k,m}=(a^{(m)}_{ij})_{1\le i,j\le k}, \mbox{ and }B_{k,m}=(b^{(m)}_{ij})_{1\le i,j\le k},
		\]
		\item Let $\vec{\mathbf 1}$ denote the all-one vector, and $\vec{\mathbf 0}$ the all-zero vector.
		\item Define $D:=\begin{bmatrix}
		A_{k-1,m} & \vec{\bf 1} \\
		\vec{\mathbf 0}^\intercal & 1
		\end{bmatrix}$.
		\item Let $W_{k-1,m}$ be the $(k-1)\times (k-1)$ matrix given by $W_{k-1,m}=(w^{(m)}_{ij})_{1\le i,j\le k-1}$, and $w^{(m)}_{ij}=(-1)^{j-i}\binom{m-k+j-i-1}{j-i}$. Define $W' := \begin{bmatrix} W_{k-1,m}& \vec{\mathbf 0}\\ \vec{\mathbf 0}^\intercal & 1\end{bmatrix}$.
	\end{itemize}
\end{notn}
First, we observe that the assertion of Lemma \ref{lem:a_to_b} can be rephrased as,
\begin{equation}
\label{eqn:g=Ba}
\vec{g}=B_{k,k}\vec{a}.
\end{equation}

Next, in terms of matrices, equality (\ref{eqn:a=sum(f)}) reads
\begin{equation}
\label{eqn:a=Df}
\vec{a}=D\vec{d}.
\end{equation}

Finally, by the definition of $\vec{d}$, we have
\begin{equation}
\label{eqn:f=w'b}
\vec{d} = W'\vec{b}.
\end{equation}

Putting together Equations (\ref{eqn:g=Ba},\ref{eqn:a=Df},\ref{eqn:f=w'b}), we obtain:
\[
\vec{g} = B_{k,k}DW'\cdot \vec{b}.
\]

By Proposition \ref{prop:apndx-BDW'=I} from the Appendix, we know that the product matrix $B_{k,k}DW'$ is $I_k$, and this concludes the proof of Lemma \ref{lem:linearmatrixupperbd}.
\hfill{$\blacksquare$}

\medskip

\begin{proof}[Proof of Theorem \ref{thm:generalupperlowerbd}]
	
We now have gathered all the equipment required to complete the proof of Theorem \ref{thm:generalupperlowerbd}. Recall that $\alpha=\min\limits_{1\le i \le k-2}\left(\frac{b_i(m)}{mb_{i+1}(m)}\right)$, and we wish to prove that
\[
f(m,H(\vec{\bf b}))\le \frac{k(k-1)}{\alpha}+k-1.
\]
Note that this bound is trivial if $\frac{k(k-1)}{\alpha}\ge m$, therefore we may assume that $\alpha m> k(k-1)$. From the definition of $\alpha$, note that $b_i\ge \alpha mb_{i+1}$ for each $1\le i\le k-2$. By successively applying these inequalities we obtain $b_i\ge \alpha m b_{i+1}\ge \alpha^2m^2b_{i+2}\ge\cdots\ge \alpha^{k-i-1}m^{k-i-1}b_{k-1}$. Thus,
\begin{equation}
\label{eqn:pfthmlongineq}
b_i\ge \alpha mb_{i+1} \ge \sum_{r=i+1}^{k-1}\frac{\alpha m}{k}\cdot b_{i+1}\ge\sum_{r=i+1}^{k-1}\frac{\alpha^{r-i}m^{r-i}}{k}\cdot b_r\ge \sum_{r=i+1}^{k-1}\left(\frac{\alpha m}{k}\right)^{r-i}b_r\ge \sum_{r=i+1}^{k-1}\binom{\left\lfloor\frac{\alpha m}{k}\right\rfloor}{r-i}b_r.
\end{equation}
The last inequality follows from $X^t\ge \binom {\lfloor X\rfloor}t$. Observe that the assumption $\frac{\alpha m}k> k-1$ implies $\left\lceil\frac{\alpha m}k\right\rceil\ge k$. Therefore, for $1\le i\le k-2$ and $i+1\le r\le k-1$, we have
\[
\left\lfloor\frac{\alpha m}k\right\rfloor\ge\left\lceil\frac{\alpha m}k\right\rceil-k+r-i-1\ge 0.
\]
Thus, (\ref{eqn:pfthmlongineq}) gives us
\[
b_i\ge\sum_{r=i+1}^{k-1} \binom{\left\lfloor\frac{\alpha m}k\right\rfloor}{r-i}b_r\ge \sum_{r=i+1}^{k-1}\binom{\left\lceil\frac{\alpha m}k\right\rceil-k+r-i-1}{r-i}b_r\ge \sum_{r=i+1}^{k-1}(-1)^{r-i+1}\binom{\left\lceil\frac{\alpha m}k\right\rceil-k+r-i-1}{r-i}b_r,
\]
implying
\[
b_i+\sum_{r=i+1}^{k-1}(-1)^{r-i}\binom{\left\lceil\frac{\alpha m}k\right\rceil-k+r-i-1}{r-i}b_r\ge 0.
\]
This is exactly the condition (\ref{eqn:cond_on_bi}), with $m$ replaced by $\left\lceil\frac{\alpha m}{k}\right\rceil$, so Lemma \ref{lem:linearmatrixupperbd} gives us a hypergraph $K$ on $\left\lceil\frac{\alpha m}{k}\right\rceil$ edges such that every $k$ sets of $K$ are isomorphic to $H(\vec{\bf b})$. 

\begin{figure}[H]
\begin{center}
    \begin{tikzpicture}
        \draw (0,0) circle [radius=1]; 
        \draw (2.5,0) circle [radius=1];
        \draw [fill] (4,0) circle [radius=0.05];
        \draw [fill] (4.5,0) circle [radius=0.05];
        \draw [fill] (5,0) circle [radius=0.05];
        \draw (6.5,0) circle [radius=1];
        \draw (0,0) node {$\displaystyle\left\lceil\frac{\alpha m}{k}\right\rceil$};
        \draw (0,1.1) node [above] {$K$};
        \draw (0,-1.1) node [below] {$1$};
        \draw (2.5,0) node {$\displaystyle\left\lceil\frac{\alpha m}{k}\right\rceil$};
        \draw (2.5,1.1) node [above] {$K$};
        \draw (2.5,-1.1) node [below] {$2$};
        \draw (6.5,0) node {$\displaystyle\left\lceil\frac{\alpha m}{k}\right\rceil$};
        \draw (6.5,1.1) node [above] {$K$};
        \draw (6.5,-1.1) node [below] {$\left\lceil{k}/{\alpha}\right\rceil$};
    \end{tikzpicture}
\caption{Constructing $F_m$ from copies of $K$}
\end{center}
\end{figure}
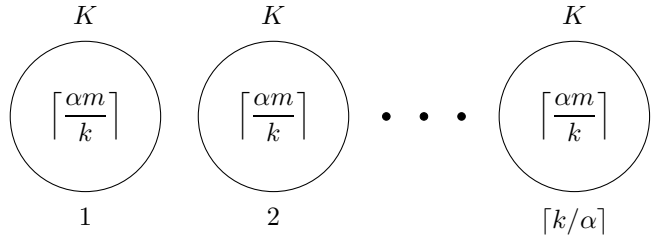

Now, consider a $\left\lceil\frac k\alpha\right\rceil$-fold disjoint union of $K$'s. This hypergraph $F_m$ has $\left\lceil\frac k\alpha\right\rceil\cdot\left\lceil\frac{\alpha m}{k}\right\rceil \ge m$ edges, and note that as long as we pick $1+\left\lceil\frac{k}{\alpha}\right\rceil\cdot(k-1)$ edges, some $k$ of them fall in the same copy of $K$. These $k$ edges create a $H(\vec{\bf b})$ by construction of $K$. This shows $f(m,H(\vec{\bf b}))\le \left\lceil\frac{k}{\alpha}\right\rceil\cdot (k-1)$, completing the proof of the upper bound. 

\bigskip

Now we prove the lower bound. Recall that we are aiming to prove
	\begin{equation}
	\label{eqn:lowerbd1}
	f(m,H(\vec{\bf b}))\ge\max_{1\le i \le k-2}\left({\frac{mb_{i+1}}{2(b_i+b_{i+1})\binom{b_{k-1}+b_k}{b_k}}}\right)^{\frac 1k}.
	\end{equation}
	
	Suppose $F$ is a hypergraph on $m$ edges. Either $F$ has a subgraph $F_1$ of size $\frac m2$ which is of the same uniformity as $H(\vec{b})$, or it has a subgraph of size $\frac m2$ which is not of the this uniformity. If the latter is true, then $\ex(F,H(\vec{b}))\ge \frac m2$. Otherwise, we focus on the subgraph $F_1$. Let $T$ be a $H(\vec b)$-free subgraph in $F_1$ of maximum size, say $|T|=t$. Then, for every $S\in F_1\setminus T$, there exist distinct $A_1,\ldots,A_{k-1}\in T$ such that $\{A_1,\ldots,A_{k-1},S\}$ forms a $H(\vec{b})$. Therefore, there are fixed $A_1,\ldots,A_{k-1}\in T$ such that $\{A_1,\ldots,A_{k-1},S\}$ forms a $H(\vec b)$ for every $S\in F_2$, where
	\[
	|F_2|\ge \frac{\frac m2-t}{\binom{t}{k-1}}.
	\]
	Further, note that $|A_1\cap\cdots\cap A_{k-1}\cap S|=b_{k}$ for every $S\in F_2$, therefore there is a subgraph $F_3\subseteq F_2$ such that every element $S\in F_3$ intersects $A_1\cap \cdots\cap A_{k-1}$ in the exact same set, and
	\[
	|F_3|\ge\frac{\frac m2-t}{\binom t{k-1}\binom{b_{k-1}+b_{k}}{b_k}}.
	\]
	Finally, for any $1\le i\le k-2$, let
	\[
	h_i:=|\{(x,B):\ x\in A_1\cap \cdots\cap A_i\setminus (A_{i+1}\cup \cdots\cup A_{k-1}),\ B\in F_3,\ x\in B\}|.
	\] 
	Let $D:= \max\limits_{x\in V(F_3)}\deg_{F_3}(x)$. As $\{A_1,\ldots,A_{k-1},B\}$ is an $H(\vec{b})$,
	\begin{equation}
	\label{eqn:doublecount1}
	|F_3|\cdot b_{i+1} = h_i\le D \cdot |A_1\cap \cdots\cap A_i\setminus (A_{i+1}\cup \cdots\cup A_{k-1})|.
	\end{equation}
	Now, for a fixed $S\in F_3$, let $X:=A_1\cap \cdots\cap A_i\setminus (A_{i+1}\cup \cdots\cup A_{k-1})$. Note that,
	\[
	\begin{aligned}|X|&=|S\cap X|+|X\setminus S|\\&=|S\cap A_1\cap \cdots\cap A_i\setminus (A_{i+1}\cup \cdots\cup A_{k-1})|+|A_1\cap \cdots\cap A_i\setminus (A_{i+1}\cup \cdots\cup A_{k-1}\cup S)| \\&=b_{i+1}+b_{i},\end{aligned}
	\]
	Therefore (\ref{eqn:doublecount1}) implies
	\[
	D\ge \frac{|F_3|\cdot b_{i+1}}{b_i + b_{i+1}}\ge \frac{(\frac m2-t)b_{i+1}}{\binom{t}{k-1}\binom{b_{k-1}+b_k}{b_k}(b_i+b_{i+1})}.
	\]
	Clearly the sets in $F_3$ that achieve the maximum degree $D$ is $H(\vec{b})$-free, leading to the inequality
	\[
	t\ge \frac{(\frac m2-t)b_{i+1}}{\binom{t}{k-1}(b_i+b_{i+1})\binom{b_{k-1}+b_k}{b_k}},
	\]
	i.e.,
	\[
	t\binom{t}{k-1}\ge \frac{(\frac m2-t)b_{i+1}}{(b_i+b_{i+1})\binom{b_{k-1}+b_k}{b_k}}.
	\]
	Since $m\ge 6$, note that if $t\ge\frac m4$, then $t\ge \left(\frac m2\right)^{\frac 13}\ge \left(\frac m2\right)^{\frac 1k}$, which is larger than the right side of (\ref{eqn:lowerbd1}). So we may assume $t<\frac m4$, which would lead us to
	\begin{equation}
	\label{eqn:lowerbd1proof}
	t^k\ge 2t\binom{t}{k-1}\ge \frac{mb_{i+1}}{2(b_i+b_{i+1})\binom{b_{k-1}+b_k}{b_k}}.
	\end{equation}
	As (\ref{eqn:lowerbd1proof}) holds for every $1\le i\le k-2$, this gives the bound that we seek.
\end{proof}


\section{Proof of Theorem \ref{thm:generallowerbd}}
In this section we prove Theorem \ref{thm:generallowerbd}. The proof is by induction on $b_k$, starting from $b_k=0$. Notice that the lower bound of Theorem \ref{thm:generalupperlowerbd} gives us the following corollary, which serves as the base case for our induction argument:

\begin{cor}
\label{cor:lowerbdb_k=0}
For $m\ge 6$,
\[
f(m,H(b_1,\ldots,b_{k-1},0))\ge \max_{1\le i\le k-2}\left({\frac{mb_{i+1}}{2(b_i+b_{i+1})}}\right)^{\frac 1k}.
\]
\end{cor}

Further, one can asymptotically improve this bound when $k=3$:
\begin{prop}
\label{prop:lowerbound_b3=0}
For $m\ge 4$,
\[
f(m,H(b_1,b_2,0))\ge \sqrt{\frac{mb_2}{2(b_1+2b_2)}}.
\]
\end{prop}
\begin{proof}
Let $|F|=m$ and $H=H(b_1,b_2,0)$. Either $F$ has a $(b_1+2b_2)$-uniform subgraph $F_1$ of size $\frac m2$, or it has a subgraph of size $\frac m2$ in which none of the edges have size $(b_1+2b_2)$. If the latter is true, then $\ex(F,H)\ge\frac m2$. Otherwise let us focus on $F_1$. Let $T$ be an $H$-free subset of maximum size in $F_1$, and suppose $|T|=t$. Note that for any $B\in F_1\setminus T$, there are sets $A_1,A_2\in T$ such that $(B,A_1,A_2)$ is a $H(b_1,b_2,0)$. Suppose $V=\bigcup_{A\in T}A$, then we have $|B\cap V|\ge 2b_2$, and $|V|\le t(b_1+2b_2)$. Let $D = \max\limits_{x\in V}\deg_{F_1}(x)$. Then,
\[
2b_2\cdot |F_1\setminus T|\le |\{(x,B):\ x\in V, B\in F_1\setminus T, x\in B\}|\le D\cdot |V|,
\]
and
\[
D\ge \frac{(m-2t)b_2}{t(b_1+2b_2)}.
\]
Let $x\in V$ have the maximum degree in $F$. Since the subgraph of size $D$ containing $x$ is $H$-free, we obtain
\[
t\ge\frac{(m-2t)b_2}{t(b_1+2b_2)}.
\]
If $t\ge\frac m4$, then $t\ge\frac12\sqrt{m}\ge \sqrt{\frac{mb_2}{2(b_1+2b_2)}}$. So assume $t<\frac m4$, and therefore $t^2\ge \frac{mb_2}{2(b_1+2b_2)}$, as desired.
\end{proof}

Before we prove Theorem \ref{thm:generallowerbd} we require the following lemma from \cite{spencer-turan1972}:
\begin{lemma}
\label{lem:alphaH}
Let $H=(V,E)$ be a $k$-graph on $m$ vertices, and let $\alpha(H)$ denote the independence number of $H$. Then, 
\[
\alpha(H)\ge\frac{k-1}{k}\cdot\left({\frac{m^k}{k|E(H)|}}\right)^{\frac1{k-1}}.
\]
\end{lemma}

Now we are prepared to prove Theorem \ref{thm:generallowerbd}.

\medskip
\begin{proof}[Proof of Theorem \ref{thm:generallowerbd}]
	
Fix $k$ and $\vec{\bf b}$. Recall that $b_k$ is fixed, and we wish to show that for $m\ge 6$,
\begin{equation}
\label{eqn:lowerbdseeked}
f(m,H(b_1,\ldots, b_k))\ge
\left\{ 
	\begin{array}{lc}
	m^{\frac 1{k(b_k+1)}}\left(\frac{b_{k-1}}{4(b_{k-2}+2b_{k-1})}\right)^{\frac 1k}, & k\ge 4,\\
	m^{\frac1{b_3+2}}\left(\frac{b_2}{4(b_1+2b_2)}\right)^{\frac{b_3+1}{b_3+2}}, & k=3.
	\end{array}
\right.
\end{equation}

Suppose $|F|=m$. Then, either $F$ has a subgraph $F_1$ of size at least $\frac m2$ which has uniformity the same as that of $H(\vec{b})$, or it does not. When the latter is true, we have $\ex(F,H(\vec{b}))\ge\frac m2$. Since $\frac m2\ge m^{\frac 14}\cdot \left(\frac 18\right)^{\frac 14}$ and $\frac m2\ge m^{\frac 12}\cdot (\frac18)^{\frac 12}$, we may assume that the former is true. We wish to show that $F_1$ contains a $H(\vec{b})$-free subgraph of large size.

\smallskip
We proceed by induction on $b_k$. Notice that we already established the results for $b_k=0$ in Corollary \ref{cor:lowerbdb_k=0} (using $b_{k-1}\le 2b_{k-1}$) and Proposition \ref{prop:lowerbound_b3=0}.

Construct a $k$-graph $G$ with vertex set $F_1$ and call $\{A_1,\ldots,A_k\}$ an edge in $G$ iff $\{A_1,\ldots,A_k\}\cong H(\vec{b})$. Clearly, $t=\alpha(G)$ is a lower bound to our problem. By Lemma \ref{lem:alphaH},
\[
k|E(G)|\ge \left(\frac{k-1}{k}\right)^{k-1}\cdot \frac{(m/2)^k}{t^{k-1}}.
\]
Given $1\le i\le k$ and $B_1,\ldots, B_i\in F_1$, denote by $\deg_G(B_1,\ldots,B_i)$ the number of edges of $G$ containing $\{B_1,\ldots,B_i\}$. As $\sum_{A_1,\ldots,A_{k-2}\in F_1}\deg_G(A_1,\ldots,A_{k-2})=\binom k2|E(G)|$, we have
\[
\sum_{A_1,\ldots,A_{k-2}\in F_1}\deg_G(A_1,\ldots,A_{k-2})\ge \frac{\binom k2}{k}\cdot \frac{(k-1)^{k-1}}{k^{k-1}}\cdot \frac{(m/2)^k}{t^{k-1}}=\frac{(k-1)^{k}}{2k^{k-1}}\cdot \frac{(m/2)^k}{t^{k-1}}.
\]
The sum on the left side has at most $\binom {m/2}{k-2}\le \frac{(m/2)^{k-2}}{(k-2)!}$ terms, therefore there exist distinct $A_1,\ldots,A_{k-2}\in F_1$ such that
\[
\deg_G(A_1,\ldots,A_{k-2})\ge \frac{(k-2)!(k-1)^{k}}{2k^{k-1}}\cdot \frac{(m/2)^2}{t^{k-1}}.
\]
Note that $\frac{(k-2)!(k-1)^{k}}{2k^{k-1}}>\frac14$ for every $k\ge 3$. Let $\B$ denote the set of all edges $B\in F_1$ which are covered by an edge through $\{A_1,\ldots,A_{k-2}\}$. Then, $|\B|^2\ge \deg_G(A_1,\ldots,A_{k-2})$, and so
\begin{equation}
\label{eqn:lowerbdB}
|\B|^2\ge \frac14\cdot \frac{(m/2)^2}{t^{k-1}}=\frac1{16}\cdot\frac{m^2}{t^{k-1}}.
\end{equation}
As $\{A_1,\ldots,A_{k-2}\}$ is a subgraph of $H(\vec b)$, we have
\[
|A_1\cap\cdots\cap A_{k-2}|=b_{k-2}+2b_{k-1}+b_k.
\]
Also, for every $B\in\B$, $\{A_1,\ldots,A_{k-2},B\}$ is a subgraph of $H(\vec b)$. Thus,
\[
|A_1\cap\cdots\cap A_{k-2}\cap B|=b_{k-1}+b_k.
\]
Now,
\[
|\B|\cdot(b_{k-1}+b_k) = |\{(x,B):\ x\in A_1\cap\cdots\cap A_{k-1}, B\in\B, x\in B\}| = \sum_{x\in A_1\cap\cdots\cap A_{k-2}} \deg_\B(x).
\]
Let $D$ be the maximum degree of a vertex in $F_1$. Then, by (\ref{eqn:lowerbdB}),
\begin{equation}
\label{eqn:lowerbdD}
D\cdot (b_{k-2}+2b_{k-1}+b_k)\ge |\B|\cdot (b_{k-1}+b_k)\ge \frac{1}{4}(b_{k-1}+b_k)\cdot\frac{m}{t^{\frac{k-1}{2}}}.
\end{equation}
Also, note that 
\[
\frac{b_{k-1}+b_k}{b_{k-2}+2b_{k-1}+b_k}\ge \frac{b_{k-1}}{b_{k-2}+2b_{k-1}}\iff b_k(b_{k-2}+b_{k-1})\ge 0.
\]
Therefore (\ref{eqn:lowerbdD}) gives us,
\begin{equation}
\label{eqn:lowerbdDmajor}
D\ge \frac{1}{4}\cdot \frac{b_{k-1}}{b_{k-2}+2b_{k-1}}\cdot \frac{m}{t^{\frac{k-1}{2}}}.
\end{equation}
Now, we notice that if $x$ is a vertex of degree $D$, then deleting it from the edges through $x$ gives us a family of uniformity one less than that of $F_1$. By induction on $b_k$, this subfamily already contains a $H(b_1,\ldots,b_{k-1},b_k-1)$-free family of size $f(D,H(b_1,\ldots,b_{k-1},b_k-1))$, which is a natural lower bound to our problem. Therefore,
\[
t\ge f(D,H(b_1,\ldots,b_{k-1},b_k-1))
\]
We now split into two cases.
\begin{itemize}
\item {\bf Case I: $k\ge 4$.} Now we use the inductive lower bound given by (\ref{eqn:lowerbdseeked}):
\[
t\ge D^{\frac 1{kb_k}}\left(\frac{b_{k-1}}{4(b_{k-2}+2b_{k-1})}\right)^{\frac 1k}\iff D\le \left(\frac{4(b_{k-2}+2b_{k-1})}{b_{k-1}}\right)^{b_k}\cdot t^{kb_k}.
\]
Combining this bound with (\ref{eqn:lowerbdDmajor}), we get
\[
\left(\frac{4(b_{k-2}+2b_{k-1})}{b_{k-1}}\right)^{b_k}\cdot t^{kb_k}\ge \frac{b_{k-1}}{4(b_{k-2}+2b_{k-1})}\cdot \frac{m}{t^{\frac{k-1}{2}}},
\]
Which, on invoking ${t^{\frac{k-1}{2}}}\le {t^k}$, leads us to
\[
t^{k(b_k+1)}\ge m\left(\frac{b_{k-1}}{4(b_{k-2}+2b_{k-1})}\right)^{b_k+1},
\]
finishing off the induction step.
\item {\bf Case II: $k=3$.} In this case we use the inductive lower bound in (\ref{eqn:lowerbdseeked}) of
\[
t\ge D^{\frac1{b_3+1}}\left(\frac{b_2}{4(b_1+2b_2)}\right)^{\frac{b_3}{b_3+1}}\iff D\le \left(\frac{4(b_1+2b_2)}{b_2}\right)^{b_3}\cdot t^{b_3+1}.
\]
Again, combining this bound with (\ref{eqn:lowerbdDmajor}), we obtain
\[
\left(\frac{4(b_1+2b_2)}{b_2}\right)^{b_3}\cdot t^{b_3+1}\ge \frac{b_{2}}{4(b_{1}+2b_{2})}\cdot \frac{m}{t}.
\]
This implies $t\ge m^{\frac1{b_3+2}}\left(\frac{b_2}{4(b_1+2b_2)}\right)^{\frac{b_3+1}{b_3+2}}$, completing the induction step.

\end{itemize}
\end{proof}


\section{Proof of Theorem \ref{thm:inversiveplane}}
In this section, we prove Theorem \ref{thm:inversiveplane}. For the proof, we rely upon the incidence structure of Miquelian inversive planes $\mathbf M(q)$ of order $q$. An inversive plane consists of a set of points $\mathcal P$ and a set of circles $\mathcal C$ satisfying three axioms \cite{dembowski-hughes-inversive1965}:
\begin{itemize}
	\item Any three distinct points are contained in exactly one circle.
	\item If $P\neq Q$ are points and $c$ is a circle containing $P$ but not $Q$, then there is a unique circle $b$ through $P,Q$ and satisfying $b\cap c=\{P\}$.
	\item $\mathcal P$ contains at least four points not on the same circle.
\end{itemize}
Every inversive plane is a $3$-$(n^2+1,n+1,1)$-design for some integer $n$, which is called its order. An inversive plane is called Miquelian if it satisfies Miquel's theorem \cite{dembowski-hughes-inversive1965}. The usefulness of Miquelian inversive planes lies in the fact that their automorphism groups are sharply 3-transitive (cf. pp 274-275, Section 6.4 of \cite{dembowski-book2012}). There are a few known constructions of $\mathbf M(q)$, one such construction is outlined here. The points of $\mathbf M(q)$ are elements of $\mathbb F_q^2$ and a special point at infinity, denoted by $\infty$. The circles are the images of the set $K=\mathbb F_q\cup\{\infty\}$ under the permutation group $PGL_2(q^2)$, given by
\[
x\mapsto \frac{ax^\alpha+c}{bx^\alpha+d},\ ad-bc\neq 0, \alpha\in\Aut(\mathbb F_q^2).
\]
For further information on inversive planes and their constructions, the reader is referred to \cite{dembowski-book2012,biliotti-montinaro2012,rinaldi-inversive2001}.

Now, we prove Theorem \ref{thm:inversiveplane}.
\begin{proof}[Proof of Theorem \ref{thm:inversiveplane}.]
Recall that for every odd prime power $q$, we are required to demonstrate a hypergraph on $q^2+1$ edges with the property that every three edges form an $H(q^2-q-1,q,1)$. Let $\mathbf M(q)$ be a Miquelian inversive plane, with points labelled $\{1,2,\ldots,q^2+1\}$. Then, we consider the $(q^2+q)$-graph $F=\{A_1,\ldots,A_{q^2+1}\}$, whose vertex set $V(F)$ is the circles of $\mathbf M(q)$, and $A_i$ is the collection of circles containing $i$. By the inversive plane axiom, any three distinct points have a unique circle through them. It suffices to show that any two distinct points $P,Q$ in $\mathbf M(q)$ have $q+1$ distinct circles through them. By 2-transitivity of the Automorphism group, we know that any two points have the same number $a_2$ of circles through them. Now, for any $P\neq Q$,
	\[
	(q^2+1-2)\cdot 1 =|\{(R,c): R\mbox{ is a point}, c\mbox{ is a circle through }P,Q,R\}|=a_2\cdot (q+1-2),
	\]
	Thus $a_2=q+1.$ So, $F$ is $(q^2+q)$-uniform, every two edges of $F$ have an intersection of size $q+1$, and every three edges of $F$ have an intersection of size $1$. By inclusion-exclusion, they form a $H(q^2-q-1,q,1)$.
\end{proof}

Now, we prove Corollary \ref{cor:extendinversive}.

\begin{proof}[Proof of Corollary \ref{cor:extendinversive}]
Recall that we wish to extend the result of the Theorem \ref{thm:inversiveplane}. We wish to show that when $b_1\ge b_2^{\,2}-b_2-1$ and $b_2^{\,2}\ge m$, $f(m,H(b_1,b_2,1))=2$.

Notice that the inversive plane construction asymptotically gives us $b_2^{\,2}+1$ sets such that any three of them are an isomorphic copy of $H(b_2^{\,2}-b_2-1,b_2,1)$. As long as $b_2^{\,2}+1\ge m$, we can take a subgraph of the construction and still obtain $m$ sets satisfying the same property. Also note that as long as $b_1\ge b_2^{\,2}-b_2-1$, we can create a construction by first creating an inversive plane $F$ satisfying $b_1=b_2^{\,2}-b_2-1$, and then adding $(b_1-b_2^{\,2}+b_2+1)$ new distinct points to each set in $F$.
\end{proof}


\section{Further Problems} \label{problems}
We discuss a few further problems that are interesting. Of course, the main open question is (\ref{eqn:mainqn2}), which asks to characterize all sequences of $k$-edge hypergraphs $H_m$ for which $f(m,H_m)$ is bounded. As we discussed, even the case $k=3$ turns out to be interesting. 

Let us focus on the case $k=3$ and $\vec{\bf b}=(b_1,b_2,1)$. The current state of affairs was summarized in Figure \ref{fig:finalpic}. An interesting observation is that all the upper bounds in the light regions are actually upper bounds of $2$. Therefore, the following question may be interesting to ask:

\begin{prob}
\label{prob:prob1}
	Characterize all values of $(b_1,b_2)$ such that $f(m,H_3(b_1,b_2,1))=2$.
\end{prob}

We cannot solve this problem completely. However, a straightforward counting argument estimating $\sum d_i$, $\sum d_i^2$ and $\sum d_i^3$, along with Cauchy-Schwarz yields the following necessary condition.
\begin{thm}
\label{thm:necessary_condition_for_f=2}
	Suppose $f(m,H_3(b_1,b_2,b_3))=2$. Then, for large enough $m$, $b_1,b_2,b_3$ must satisfy
	\[
	b_1b_3+\frac{b_1b_2}{m}+\frac{b_2b_3}{m}\ge b_2^{\,2}.
	\]
	In particular, when $b_3=1$, these numbers satisfy
	\[
	b_1+\frac{b_1b_2}m\ge b_2^{\,2}.
	\]
\end{thm}

Theorem \ref{thm:necessary_condition_for_f=2} gives more insight into Figure \ref{fig:finalpic}. Basically, there are two cases to consider. When $b_1$ is asymptotically larger than $\frac{b_1b_2}{m}$, ie when $b_2=o(m)$, this means that $b_1\ge b_2^{\,2}$ is necessary for $f=2$. When $b_2\ge m$, this gives us $b_1\ge mb_2$, which is exactly the construction in Lemma \ref{lem:linearmatrixupperbd}. Further, note that this transition occurs exactly at the intersection of the line $b_1=mb_2$ and the parabola $b_1=b_2^{\,2}$.
\medskip

As a further special case of Problem \ref{prob:prob1}, one can look at $\vec{\bf b} = (m,b_2,1)$ where $1\ll b_2\ll \sqrt m$. We expect this range to be solvable via a construction, since there are constructions for $b_2=1$ (Theorem \ref{thm:generalupperlowerbd}) and $b_2=\sqrt m$ (Theorem \ref{thm:inversiveplane}). The problem is equivalent to constructing bipartite graphs with certain properties, as stated below.

\begin{prob}
\label{prob:prob2}
	Suppose $1\ll b_2\ll\sqrt m$. Is there a bipartite graph $G$ with parts $A$, $B$, such that $|A|=m$, the degree of every vertex in $A$ is asymptotic to $m$, the size of the common neighborhood of every pair in $A$ is asymptotic to $b_2$, and every three vertices in $A$ have a unique common neighbor in $B$?
\end{prob}

If such a bipartite graph can be constructed, then we can let $F = \{N_G(u): u\in A\}$. This hypergraph will testify for $f(m,H(m,b_2,1))=2$. From the proof of Theorem \ref{thm:necessary_condition_for_f=2}, we know that if such a bipartite graph exists, it cannot be regular from $B$. A regular construction from $B$ implies equality in the Cauchy-Schwarz inequality, which would imply $b_2=\Theta(\sqrt m)$. Therefore if such a graph is constructed, $B$ needs to have vertices of different degrees.

\smallskip
Notice also, that if the answer to Problem \ref{prob:prob2} is affirmative, then we can shade the small triangle in Figure \ref{fig:finalpic} light. This is courtesy the fact that given any $(b_1,b_2)$ in that region, we can write it as a sum of $(x,y)+(m,z)$, with $x\ge my$. We can then take a $H_3(x,y,0)$-construction $\{A_1,\ldots, A_m\}$ and a $H_3(m,z,1)$-construction $\{A'_1,\ldots,A'_m\}$, and merge them together to obtain the $H_3(b_1,b_2,1)$-construction $\{A_1\cup A'_1,\ldots, A_m\cup A'_m\}$.


\section{Appendix}
Our goal in this section is to prove the matrix identity asserted in Proposition \ref{prop:apndx-BDW'=I}. Before seeing the proof, we mention an useful binomial identity in Lemma \ref{lem:vdm_convolution}.

\begin{lemma}
	\label{lem:vdm_convolution}
	For integers $x\ge 0, y\ge z\ge 0$, we have
	\begin{equation}
	\label{eqn:vdm_convolution}
	\sum_{t=0}^z(-1)^t\binom{x}{t}\binom{y-t}{z-t}=(-1)^z\binom{x-y+z-1}{z}.
	\end{equation}
\end{lemma}
\begin{proof}
	One can prove this identity using induction on $y$. Note that when $y=z$, the identity becomes
	\[
	\sum_{t=0}^z(-1)^t\binom xt = (-1)^z\binom{x-1}{z},
	\]
	which follows from applying Pascal's identity $\binom{x}{t}=\binom{x-1}{t}+\binom{x-1}{t-1}$ to each term and telescoping.
	
	Now suppose that (\ref{eqn:vdm_convolution}) holds for some $y$. Then,
	\[
	\sum_{t=0}^z(-1)^t\binom{x}{t}\binom{y-t+1}{z-t} = \sum_{t=0}^z(-1)^t\binom{x}{t}\binom{y-t}{z-t} + \sum_{t=0}^{z-1}(-1)^t\binom{x}{t}\binom{y-t}{z-t-1}.
	\]
	By induction hypothesis, the first term is $(-1)^z\binom{x-y+z-1}{z}$ and the second term is $(-1)^{z-1}\binom{x-y+z-2}{z-1}$. Their sum is $(-1)^{z}\binom{x-y+z-2}{z}$, as desired.
\end{proof}

Notice that in some cases $x-y+z-1$ could be negative, whence we interpret the binomial coefficient $\binom {-a}s$ as $(-1)^s\binom{a+s-1}{s}$. Also observe that the generalized binomial coefficients satisfy Pascal's identity $\binom{a}{s}=\binom{a-1}{s}+\binom{a-1}{s-1}$.

We are now going to state and prove Proposition \ref{prop:apndx-BDW'=I}. Recall the following notation: $a^{(m)}_{ij}=\binom{m-i}{j-i}$, $b^{(m)}_{ij}=(-1)^{j-i}\binom{m-i}{j-i}$, $w^{(m)}_{ij}=(-1)^{j-i}\binom{m-k+j-i-1}{j-i}$,
\[
A_{k,m}=(a^{(m)}_{ij})_{1\le i,j\le k}, \ B_{k,m}=(b^{(m)}_{ij})_{1\le i,j\le k}, W_{k-1,m}=(w^{(m)}_{ij})_{1\le i,j\le k-1},
\]
and,
\[D = \begin{bmatrix} A_{k-1,m} & \vec{\bf 1} \\ \vec{\mathbf 0}^\intercal & 1 \end{bmatrix},\ 
W'=\begin{bmatrix} W_{k-1,m}& \vec{\mathbf 0}\\ \vec{\mathbf 0}^\intercal & 1\end{bmatrix}.
\]

\begin{prop}
\label{prop:apndx-BDW'=I}
\[
B_{k,k}\cdot D\cdot W' = I_k.
\]
\end{prop}

\begin{proof}
	Note that $B_{k,k}=\begin{bmatrix}B_{k-1,k}&\vec v\\\vec{\bf 0}^\intercal& 1\end{bmatrix}$, where $v_i=(-1)^{k-i}$, and therefore
	\[
	B_{k,k}DW' = \begin{bmatrix}
	B_{k-1,k}A_{k-1,m}W_{k-1,m} & B_{k-1,k}\vec{\bf 1}+v\\
	\vec{\bf 0}^\intercal & 1
	\end{bmatrix}.
	\]
	We verify that $B_{k-1,k}\vec{\bf 1}+v=\vec{\bf 0}$ and $B_{k-1,k}A_{k-1,m}W_{k-1,m}=I_{k-1}$ in Claims \ref{clm:B1+v=0} and \ref{clm:BAW=I}, respectively.
	\begin{claim}
	\label{clm:B1+v=0}
		$B_{k-1,k}\vec{\bf 1}+v = \vec{\bf 0}.$
	\end{claim}
	\begin{proof}[Proof of Claim \ref{clm:B1+v=0}]
		Note that the $i$'th row of $B_{k-1,k}\vec{\bf 1}$ is
		\[
		\begin{aligned}
		\sum_{j=1}^{k-1}b_{ij}^{(m)} = \sum_{j=i}^{k-1}(-1)^{j-i}\binom{k-i}{j-i}= \sum_{j=0}^{k-i-1}(-1)^{j}\binom{k-i}{j}=0 - (-1)^{k-i} = -v_i,
		\end{aligned}
		\]
		as desired.
	\end{proof}
	\begin{claim}
	\label{clm:BAW=I}
		$B_{k-1,k}A_{k-1,m}W_{k-1,m}=I_{k-1}$.
	\end{claim}
	\begin{proof}[Proof of Claim \ref{clm:BAW=I}]
		Note that the $(i,j)$th entry of the product matrix is given by
		\begin{equation}
		\label{eqn:propeq1}
		\begin{aligned}
		\sum_{r=1}^{k-1}\sum_{s=1}^{k-1}b^{(k)}_{ir}a^{(m)}_{rs}w^{(m)}_{sj}&=\sum_{r=1}^{k-1}\sum_{s=1}^{k-1}(-1)^{r-i+j-s}\binom{k-i}{r-i}\binom{m-r}{s-r}\binom{m-k+j-s-1}{j-s}\\
		&=\sum_{s=1}^{k-1}(-1)^{j-s}\binom{m-k+j-s-1}{j-s}\sum_{r=1}^{k-1}(-1)^{r-i}\binom{k-i}{r-i}\binom{m-r}{s-r}.
		\end{aligned}
		\end{equation}
		Observe that, using Lemma \ref{lem:vdm_convolution} for $x=k-i, y=m-i, z=s-i$, we get
		\[
		\begin{aligned}
		\sum_{r=1}^{k-1}(-1)^{r-i}\binom{k-i}{r-i}\binom{m-r}{s-r}=\sum_{r=i}^{s}(-1)^{r-i}\binom{k-i}{r-i}\binom{m-r}{s-r}&=\sum_{r=0}^{s-i}(-1)^r\binom{k-i}{r}\binom{m-i-r}{s-i-r}\\&=(-1)^{s-i}\binom{k-m+s-i-1}{s-i}.
		\end{aligned}
		\]
		Plugging this back into (\ref{eqn:propeq1}), we get that the $(i,j)$th entry of the product matrix is
		\begin{equation}
		\label{eqn:propeq2}
		\sum_{s=1}^{k-1}(-1)^{j-i}\binom{m-k+j-s-1}{j-s}\binom{k-m+s-i-1}{s-i}
		\end{equation}
		Notice that the sum in (\ref{eqn:propeq2}) only runs from $s=i$ to $s=j$, and therefore after the change of variable $s\mapsto s+i$, the expression reduces to
		\begin{equation}
		\label{eqn:propeq3}
		(-1)^{j-i}\sum_{s=0}^{j-i}\binom{m-k+j-s-i-1}{j-i-s}\binom{k-m+s-1}{s}.
		\end{equation}
		Note that $\binom{s-(m-k)-1}{s}=(-1)^s\binom{m-k}s$, so (\ref{eqn:propeq3}) is the sum
		\[
		(-1)^{j-i}\sum_{s=0}^{j-i}(-1)^s\binom{m-k}s\binom{m-k+j-i-1-s}{j-i-s},
		\]
		which, on invoking Lemma \ref{lem:vdm_convolution} for $x=m-k, y=m-k+j-i-1, z=j-i$, reduces to
		\[
		(-1)^{j-i}\cdot (-1)^{j-i}\cdot \binom{m-k-m+k-j+i+1+j-i-1}{j-i} = \binom{0}{j-i}.
		\]
		Clearly, this is $0$ when $j\neq i$ and $1$ when $j = i$.
	\end{proof}
	This completes the proof of Proposition \ref{prop:apndx-BDW'=I}.
\end{proof}


\bibliographystyle{unsrt}
\bibliography{Hfree.bbl}

\end{document}